\documentclass[a4paper,10pt]{article}

\usepackage[T1]{fontenc} 
\usepackage{amsmath,amsthm}
\usepackage[francais,english]{babel}
\usepackage[dvips]{graphicx}
\usepackage{amsfonts,amssymb}
\usepackage{geometry}
\usepackage{hyperref}
\usepackage{stmaryrd}
\usepackage{fancyhdr}
\usepackage{url}
\usepackage{dsfont}
 \usepackage[utf8]{inputenc} 

\newtheorem*{conj*}{Conjecture}

\newtheorem*{thm*}{Theorem}

\newtheorem{prop}{Proposition}[section]

\newtheorem{LM}{Lemma}[section]

\newtheorem{thm}{Theorem}[section]

\newtheorem{cor}{Corollary}[section]

\newtheoremstyle{pourlesremarques}{\topsep}{\topsep}{\normalfont}{}{\bfseries}{.}{ }{}
\theoremstyle{pourlesremarques}

\newtheorem*{rem*}{Remark}
\newtheoremstyle{pourlesexemples}{\topsep}{\topsep}{\normalfont}{}{\bfseries}{.}{ }{}
\theoremstyle{pourlesexemples}

\renewcommand{\o}{\mathfrak{O}}

\renewcommand{\l}{\lambda}
\newcommand{\C}{\mathbb{C}}

\newcommand{\Z}{\mathbb{Z}}
\newcommand{\1}{\mathbf{1}}
\newcommand{\sm}{\mathcal{C}^\infty}
\newcommand{\D}{\Delta}

\title {\textbf{Linear and Shalika local periods for the mirabolic group, and some consequences}}
\author{Nadir MATRINGE\footnote{Nadir Matringe, Universit\'e de Poitiers, Laboratoire de Math\'ematiques et Applications,
T\'el\'eport 2 - BP 30179, Boulevard Marie et Pierre Curie, 86962, Futuroscope Chasseneuil Cedex. Email: Nadir.Matringe@math.univ-poitiers.fr}}

\begin{document}
\maketitle

\begin{abstract}
Using linear periods on the mirabolic subgroup of $GL(n,F)$, for $F$ a non archimedean local field, we give a list of the maximal Levi subgroups of $GL(n,F)$ which can distinguish discrete series and generic representations. We also obtain 
the functional equation of the local exterior-square $L$-function of a generic representation of $GL(n,F)$ when $n=2m$ is even. Then we discuss the relation between Shalika models and local models for representations of $GL(2m,F)$. Finally we give a necessary and sufficient condition on the cuspidal representation $\rho$ for a generalised Steinberg representation $St_k(\rho)$ of $GL(2m,F)$ to have a local model.
\end{abstract}

\section{Introduction}

Let $F$ be a nonarchimedean local field. We use the results of \cite{M}, which allow us to understand the distinction of representations of the mirabolic subgroup $P_n$ of $G_n=GL(n,F)$, with respect to a Levi of 
$G_n$ intersected with $P_n$, to study various problems.\\

Let $G$ be a locally compact totally disconnected group, and $H$ a closed 
subgroup of $G$. When we say a representation of $G$, we mean a complex representation, which is smooth (i.e., such that every vector in the space of the representation, is fixed by an open subgroup of $G$). We say that a representation $(\pi,V)$ of $G$ is $H$-distinguished if there is a nonzero linear form on $V$, which is fixed under $H$. A particularly interesting situation is when $H$ is a subgroup of $G$ fixed by an involution. Distinguished representations of reductive $p$-adic groups are of course of importance for the harmonic analysis of the homogeneous space $G/H$, but they are also related (in the Langlands program) with arithmetical problems, for example the occurrence of singularities of local $L$-functions associated with irreducible representations of $G$.\\

 In this paper, we study the pair $G_n$, $H=L$, with $L$ a maximal Levi subgroup of $G_n$. We first study the 
$L$-distinguished representations $\D$ which embed in the regular representation $L^2(G/Z)$, where $Z$ 
is the center of $G$ (the so called discrete series representations). We actually show that no such representations exist unless $n$ is even, and $L\simeq G_{n/2}\times G_{n/2}$ (Theorem \ref{goodlevidiscr}).
Then we study a more general (and very important) class of representations of $G_n$, the so-called generic representations, and show in this case that if a generic representation of $G_n$ is $L$-distinguished, then 
$L\simeq G_{n/2}\times G_{n/2}$ when $n$ is even, and $L\simeq G_{[n/2]+1}\times G_{[n/2]}$ when $n$ is odd ((Theorem \ref{goodlevigen}). Another application of the results of \cite{M} which is given here is the local functional equation of the exterior-square $L$-function of a generic representation $\pi$ of $G_n$, for $n$ even. This result has been obtained by a global argument for generic representations $\pi$ occuring as local components of cuspidal automorphic representations in \cite{KR}. Here the proof is local and works with no restriction on the generic representation $\pi$. It is a result of \cite{JR} that when $n$ is even, irreducible representations of $G_n$ admitting a Shalika functional, are $G_{n/2}\times G_{n/2}$-distinguished. A converse is obtained in \cite{JNQ} for cuspidal representations, and actually extends (Theorem \ref{local-shalika}, which is a theorem from Sakellaridis and Venkatesh, called the unfolding principle) to a larger class of representations, namely the relatively (square)-integrable representations of  $G_n$ with respect to $G_{n/2}\times G_{n/2}$. Now the existence of a Shalika model for a discrete series $\D$ of a representation is equivalent to the occurrence of a pole at zero in the exterior-square $L$-function of $\D$ (Proposition \ref{poledist}), defined by Jacquet and Shalika. But this $L$-function is known to be equal to the exterior-square $L$-function of the Langlands parameter of $\D$ according to the main result of \cite{KR}. Writing $\D$ as a generalised Steinberg representation $St_k(\rho)$, with $\rho$ a cuspidal representation of 
a smaller general linear group, this allows us to give necessary and sufficient conditions on $\rho$ 
for $\D$ to have a Shalika, or equivalently a local model (Theorem \ref{shaldiscr}).\\

Finally, some results of this paper were stated without proofs in \cite{CP}. This preprint 
was obviously a source of inspiration for the following work.

\section{Preliminaries}

We start with the groups at stake in the paper. Let $F$ be a nonarchimedean local field, with normalised absolute value $|.|$. We denote $GL(n,F)$ by $G_n$ for $n\geq 1$, we will denote $|det(g)|$ by $|g|$ for a matrix in $G_n$. The group $N_n$ will be the unipotent radical of the standard Borel subgroup $B_n$ of $G_n$ given by upper triangular matrices. 
For $n\geq2$ we denote by $U_n$ the group of matrices $u(x)=\begin{pmatrix} 
                                                                                         I_{n-1}    & x\\
                                                                                                    & 1 \end{pmatrix}$ 
                                                                                                    for $x$ in $F^{n-1}$.\\

For $n> 1$, the map $g\mapsto \begin{pmatrix} g & \\ & 1 \end{pmatrix}$ is an embedding of the group $G_{n-1}$ in $G_{n}$, we denote by $P_n$ the subgroup $G_{n-1}U_n$ of $G_n$. We fix a nontrivial character $\theta$ of $(F,+)$, and denote by $\theta$ again the character $n\mapsto \theta(\sum_{i=1}^{n-1}n_{i,i+1})$ of $N_n$.
The normaliser of $\theta_{|U_n}$ in $G_{n-1}$ is then $P_{n-1}$. Suppose $n=p+q$, with $p\geq q\geq 1$, we denote by $M_{(p,q)}$ the standard Levi of $G_n$ given by matrices 
$\begin{pmatrix} h_p &  \\  & h_q  \end{pmatrix}$ with $h_p\in G_p$ and $h_q\in G_q$, and by 
$M_{(p,q-1)}$ the standard Levi of $G_{n-1}$ given by matrices 
$\begin{pmatrix} h_p &  \\  & h_{q-1}  \end{pmatrix}$ with $h_p\in G_p$ and $h_{q-1}\in G_{q-1}$. 
We denote by $M_{(p-1,q-1)}$ the standard Levi of $G_{n-2}$ given by matrices 
$\begin{pmatrix} h_{p-1} &  \\  & h_{q-1}  \end{pmatrix}$ with $h_{p-1}\in G_{p-1}$ and $h_{q-1}\in G_{q-1}$. Let $w_{p,q}$ be the permutation matrix of $G_n$ corresponding to the permutation 
$$\left(\begin{array}{llllllllllllllll} 
1 & \!\! \dots \!\! & p-q & p-q +1 & p-q +2 &\!\! \dots \!\! & p-1    & p      & p+1      &\!\! \dots \!\!  & p+q-2 & p+q-1 & p+q \\
1 & \!\! \dots \!\! & p-q & p-q +1 & p-q +3 &\!\!  \dots \!\!& p+q-3 & p+q-1 & p-q + 2 &\!\! \dots \!\!& p+q-4 & p+q-2 & p+q
\end{array}\right)$$
Let $w_{p,q-1}$ be the permutation matrix of $G_{n-1}$ corresponding to the permutation $w_{p,q}$ restricted to $\{1,\dots,n-1\}$:
$$\left(\begin{array}{lllllllllllllll} 
1 & \!\! \dots \!\! & p-q & p-q +1 & p-q +2 &\!\! \dots \!\! & p-1    & p      & p+1      &\!\! \dots \!\!  & p+q-2 & p+q-1  \\
1 & \!\! \dots \!\! & p-q & p-q +1 & p-q +3 &\!\!  \dots \!\!& p+q-3 & p+q-1 & p-q + 2 &\!\! \dots \!\!& p+q-4 & p+q-2 
\end{array}\right)$$ 
Let $w_{p-1,q-1}$ be the permutation matrix of $G_{n-2}$ corresponding to the permutation
$$\left(\begin{array}{lllllllllllllll} 
1 & \!\! \dots \!\! & p-q & p-q+1  & p-q +2 &\!\! \dots \!\! & p-2   & p-1      & p        &\!\! \dots \!\!&  p+q-3 & p+q-2   \\
1 & \!\! \dots \!\! & p-q & p-q+1  & p-q +3 &\!\!  \dots \!\!& p+q-5 & p+q-3   & p-q + 2  &\!\! \dots \!\!&  p+q-4 & p+q-2 
\end{array}\right)$$ 

We denote by $H_{p,q}$ the subgroup $w_{p,q}M_{(p,q)}w_{p,q}^{-1}$ of $G_n$, by $H_{p,q-1}$ the subgroup $w_{p,q-1}M_{(p,q-1)}w_{p,q-1}^{-1}$ of $G_{n-1}$, and by $H_{p-1,q-1}$ the subgroup $w_{p-1,q-1}M_{(p-1,q-1)}w_{p-1,q-1}^{-1}$ of $G_{n-2}$.\\
Considering $G_{n-1}$ (resp. $G_{n-2}$) as a subgroup of $G_n$ (resp. $G_{n-1}$) given by matrices of the form 
$\begin{pmatrix} g & \\ & 1 \end{pmatrix}$, one has $H_{p,q}\cap G_{n-1}=H_{p,q-1}$, and 
$H_{p,q-1}\cap G_{n-2}=H_{p-1,q-1}$.\\

Now we recall basic definitions concerning distinguished representations of $l$-groups. When $G$ is an $l$-group (locally compact totally disconnected group), and we denote by $Alg(G)$ the category of smooth complex $G$-modules. If $(\pi,V)$ belongs to $Alg(G)$, $H$ is a closed subgroup of $G$,
 and $\chi$ is a character of $H$, %we denote by $V(H,\chi)$ the subspace of $V$ generated by vectors of the form $\pi(h)v-\chi(h)v$ for $h$ in $H$ and $v$ in $V$. 
%This space is actually stable under the action of the subgroup $N_G(\chi)$ of the normalizer $N_G(H)$ of $H$ in $G$, which fixes $\chi$.\\
we denote by $\delta_H$ the positive character of $N_G(H)$ such that if $\mu$ is a right Haar measure on $H$, and $int$ is the action of $N_G(H)$ given by $(int(n)f)(h)=f(n^{-1}hn)$ on smooth functions $f$ with compact support on $H$, then $\mu \circ int(n)= \delta_H(n)\mu $ for $n$ in $N_G(H)$.\\ 
%The space $V(H,\chi)$ is $N_G(\chi)$-stable. Thus, if $L$ is a closed-subgroup of $N_G(\chi)$, and $\mu$ is a (smooth) character of $L$, the quotient $V_{H,\chi}=V/V(H,\chi)$ (that we simply denote by $V_H$ when 
%$\chi$ is trivial) becomes a smooth $L$-module for the action $l.(v + V(H,\chi))= \mu(l)\pi(l)v + V(H,\chi)$ of $L$ on 
%$V_{H,\chi}$.\\  
If $H$ is a closed subgroup of an $l$-group $G$, and $(\rho,W)$ belongs to $Alg(H)$, we define the object 
$(ind_H^G(\rho), V_c=ind_H^G(W))$ and $(Ind_H^G(\rho), V=Ind_H^G(W))$ of $Alg(G)$ 
as follows.
 The space $V$ is the space of smooth functions from $G$ to $W$, fixed under right translation by the elements of a compact open subgroup 
$U_f$ of $G$, satisfying $f(hg)=\rho(h)f(g)$ for all $h$ in $H$ and $g$ in $G$. 
The space $V_c$ is the subspace of $V$, consisting of functions with support compact mod $H$. The action of $G$ is by right translation on the functions. 
We denote by $\mathbf{1}_H$ the trivial character of $H$ (which we simply write $\mathbf{1}$ if $H=G_0$ is the trivial group). We say that a representation $\pi$ of $G$ is $H$-distinguished if the space $Hom_{H}(\pi,\1_H)$ is nonzero. More generally, if $\chi$ is a character of $H$, 
we say that $\pi$ is $(H,\chi)$-distinguished if $Hom_{H}(\pi,\chi)$ is nonzero.\\

 We will need the work of Bernstein and Zelevinsky concerning the classification of irreducible representations of $G_n$. We first define the following functors following \cite{BZ}:

\begin{itemize}

\item The functor $\Phi^{+}$ from $Alg(P_{k-1})$ to $Alg(P_{k})$ such that, for $\pi$ in $Alg(P_{k-1})$, one has
$\Phi^{+} \pi = ind_{P_{k-1}U_k}^{P_k}(\delta_{U_k}^{1/2}\pi \otimes \theta)$.

\item The functor $\Psi^{+}$ from $Alg(G_{k-1})$ to $Alg(P_{k})$, such that for $\pi$ in $Alg(G_{k-1})$, one has
$\Psi^{+} \pi = ind_{G_{k-1}U_k}^{P_k}(\delta_{U_k}^{1/2}\pi \otimes 1)=\delta_{U_k}^{1/2}\pi \otimes 1 $.

\end{itemize}

A discrete series representation $\D$ of $G_n$ is an irreducible representation such that $\chi \otimes\D$ has a coefficient which is square-integrable mod the center $Z_n$ of $G_n$, for some character $\chi$ of $F^*$. We recall (\cite{Z}, theorem 9.3) that if $\D$ is a discrete series of $G_n$, there is a unique couple $(k,r)$ of positive integers such that $kr=n$, and a unique (up to isomorphism) cuspidal representation $\rho$ of $G_r$, such that $\D$ is the 
only irreducible quotient of the representation $|.|^{(1-k)/2}\rho\times \dots \times|.|^{(k-1)/2}\rho$ of $G_n$, obtained as the normalised parabolically induced representation of the representation $|.|^{(1-k)/2}\rho\otimes \dots \otimes|.|^{(k-1)/2}\rho$ of 
the standard Levi of $G_n$ isomorphic to $(G_r)^k$. We then denote $\D=St_k(\rho)$. If convenient, we will sometimes denote $[\rho,\dots,|.|^k\rho]=|.|^{k/2}St_{k+1}(\rho)$. By definition, if $\D=[\rho,\dots,|.|^k\rho]$ and $l$ is an integer in $\{0,\dots,n\}$, we denote by $\D^{(l)}$ the representation of $G_{n-l}$ which is zero if $l$ is not a multiple of $r$, and which is $[|.|^a\rho,\dots,|.|^k\rho]$ if $l=ar$ ($\D^{(n)}=\1$).\\ 

A generic representation $\pi$ of $G_n$ is by definition an irreducible representation which admits an embedding in the compactly induced representation $ind_{N_n}^{G_n}(\theta)$, in which case it is known that the embedding is unique up to scaling by elements of $\C^*$. One denotes by $W(\pi,\theta)$ the image of this 
embedding and calls it the Whittaker model of $\pi$. Theorem 9.7 of \cite{Z} then asserts that a generic representation $\pi$ of $G_n$ is of the form $\D_1\times\dots\times \D_t$, where 
the $\D_i$'s are unique unlinked (see \cite{Z} for the definition) discrete series of $G_{n_i}$, with $\sum_i n_i=n$. We state the following result, which is a consequence of Subsection 3.5 and Lemma 4.5 of \cite{BZ}, and Proposition 9.6 of \cite{Z}:

\begin{prop}\label{filtr}
Let $\rho$ be a cuspidal representation of $G_r$, and $n=kr$ for $k\geq 1$, then the restriction $(St_k(\rho)_{|P_n},V)$ has a filtration $\{0\}=V_0\subset V_1\subset \dots \subset V_{k-1} \subset V_k=V$ with 
$V_{k-i+1}/V_{k-i}\simeq (\Phi^+)^{ir-1} \Psi^+(|.|^{i/2}St_{k-i}(\rho))$ for $i$ in $\{1,\dots,k\}$, and with the convention that $St_{0}(\rho)$ is $\mathbf{1}$. Let $\pi=\D_1\times\dots\times\D_t$ be a representation of $G_n$ which is a product of discrete series, then its restriction to $P_n$ has a filtration, 
in which each factor is of the form $(\Phi^+)^{n-k-1} \Psi^+(\tau)$, for $k\leq n-1$, and 
$\tau$ is a representation  of $G_k$ of the form $\D_1^{(n_1)}\times\dots\times\D_t^{(n_t)}$ for integers $n_i$ which are $\geq 0$, moreover $(\Phi^+)^{n-1} \Psi^+(\1)$ appears with multiplicity exactly $1$.
\end{prop}

We will be concerned with representations of $G_{2m}$ admitting local models, and Shalika models. 
If $P_{(m,m)}$ is the standard parabolic subgroup of $G_{2m}$ corresponding to the partition $(m,m)$, with Levi subgroup $M_{m,m}$, and 
unipotent radical $N_{m,m}$.
Call $S_{2m}$ the subgroup of $P_{(m,m)}$ of matrices of the form $\begin{pmatrix} g & X \\ 0 & g \end{pmatrix}$, isomorphic to a semi-direct
 product $G_m\ltimes N_{m,m}$. 
Any non-trivial character $\theta$ of $F$ defines a character $\Psi$ of $S_{2m}$ by 
$\Theta\begin{pmatrix} g & X \\ 0 & g \end{pmatrix}=\theta(Tr(g^{-1}X))$.
 A representation $(\pi,V)$ of $G_{2m}$ is said to admit a Shalika model if it is $(S_{2m},\Theta)$-distinguished for a
 (equivalently any) non-trivial character $\theta$ of $F$, 
the corresponding invariant linear form being called a Shalika functional. %We will say that a representation $(\tau,V)$ of $P_{2m}$ has a Shalika model if it is $(S_{2m}\cap P_{2m},\Psi)$-distinguished for a non-trivial character $\psi$ of $F$. 
A representation $(\pi,V)$ of $G_{2m}$ is said to admit a local (or linear) model if it is $M_{m,m}$-distinguished, 
and the corresponding invariant linear form is called a linear period.\\ 
%We will say that a representation $(\tau,V)$ of $P_{2m}$ has a local model if it is $(M_{m,m}\cap P_{2m})$-distinguished.\\

We now recall some facts about the Jacquet-Shalika exterior-square $L$-function of a generic representation of $G_n$, and the local Langlands correspondance for $G_n$.\\
%Let $\pi$ be a generic representation of $G_n$, it means that it is irreducible and admits an embedding in the compactly induced representation $ind_{N_n}^{G_n}(\theta)$, in which case it is known that the embedding is unique up to scaling by elements of $\C^*$. One denotes $W(\pi,\theta)$ the image of this 
%embedding and calls it the Whittaker model of $\pi$. It is a consequence of the work of Bernstein and Zelevinsky (see \cite{Z}) that discrete series of $G_n$ are generic. \\
Let $\pi$ be a generic representation of $G_{2m}$. With $n=2m$, it is shown in \cite{JR} the following integrals converge absolutely for $Re(s)$ greater than a real $r$ depending on $\pi$:
$$J(W,\phi,s)= \int_{N_{n}\backslash G_{n}}\int_{Lie(B_m^{-})}W(w_{m,m} \begin{pmatrix} g & X \\ & g\end{pmatrix}w_{m,m}^{-1})
\theta^{-1}(Tr\ X) dX \phi(\epsilon_m g) |g|^{s} dg$$ for $W$ in $W(\pi,\theta)$, $\phi$ in $\sm_c(F^m)$, and 
$\epsilon_m$ the row vector $(0,\dots,0,1)$ of $F^m$. They define holomorphic functions on this half-plane, having meromorphic continuation as rational functions of $q^{-s}$. Moreover, the vector space of these rational functions is actually a fractional ideal of $\C[q^{-s},q^{s}]$, which is generated by a unique local Euler factor (the inverse of a polynomial $P(q^{-s})$ with $P(0)=1$, see for example Proposition 3.2 of \cite{MY} for this point) which is denoted by $L(\pi,\wedge^2,s)$, and is called the Jacquet-Shalika exterior square $L$-function of $\pi$.\\

Let $W_F$ be the Weil group of $F$ (see \cite{BH}, Chapter 29), and let $W'_F=W_F\times SL(2,\C)$ be the Weil-Deligne group of $F$.
 We recall that according to a theorem by Harris and Taylor/Henniart, there is a natural bijection from the set of isomorphism classes of irreducible representations of $G_n$ to the the set of isomorphism classes of $n$-dimensional semi-simple representations of $W'_F$, called the Langlands 
correspondance, which we denote by $\phi$. 
Let us denote by $Sp(k)$ the (up to isomorphism) unique algebraic $k$-dimensional irreducible representation of $SL(2,\C)$. Then any finite dimensional semi-simple representations of $W'_F$ is a direct sum of representations of the form $\tau \times Sp(k)$, where $\tau$ is an irreducible representation of $W_F$. We take for definition 
of $L(\tau \times Sp(k),s)$ the definition (31.3.31) of \cite{BH}, where $\tau \times Sp(k)$ corresponds to the Weil-Deligne representation of $W_F$. We extend the definition to all semi-simple representations of $W'_F$ by putting $L(\tau\oplus \tau',s)=L(\tau,s)L(\tau',s)$. Then by definition, the exterior-square $L$-function $L(\tau,\wedge^2,s)$ of a finite dimensional 
semi-simple representations $\tau$ of $W'_F$, is $L(\wedge^2(\tau),s)$, where $\wedge^2(\tau)$ is the 
exterior-square of $\tau$.

\section{Levi subgroups of $G_{n}$ with respect to which discrete series and generic representations can be distinguished}

We recall the following proposition which follows from Propositions $2.1$ and $2.2$ of \cite{M} (in which one has injections instead of isomorphisms, but they are actually isomorphisms):

\begin{prop}\label{rec}
Let $\sigma$ belong to $Alg(P_{n-1})$, and $\chi$ be a character of $P_n\cap H_{p,q}$. Then there is a positive character $\chi_{p,q}$ of 
$P_{n-1}\cap H_{p,q-1}$, independent of $\sigma$, such that 
$$Hom_{P_n\cap H_{p,q}}(\Phi^+ \sigma,\chi)\simeq Hom_{P_{n-1}\cap H_{p,q-1}}(\sigma,\chi\chi_{p,q}).$$
Let $\sigma'$ belong to $Alg(P_{n-2})$, and $\chi'$ be a character of $P_{n-1}\cap H_{p,q-1}$. Then there is a positive character $\chi_{p,q-1}$ of 
$P_{n-2}\cap H_{p-1,q-1}$ independent of $\sigma'$, such that $$Hom_{P_{n-1}\cap H_{p,q-1}}(\Phi^+ \sigma',\chi')\simeq
 Hom_{P_{n-2}\cap H_{p-1,q-1}}(\sigma,\chi'\chi_{p,q-1}).$$
\end{prop}

It has as the following first corollary:

\begin{cor}\label{homzero}
Let $n\geq 3$, and $p$ and $q$ two integers with $p+q=n$ and $p-1\geq q\geq 0$, and $\chi$ be a character of 
$P_n\cap H_{p,q}$, then one has $Hom_{P_n\cap H_{p,q}}((\Phi^+)^{n-1}\Psi^+(\mathbf{1}),\chi)=0$.
\end{cor}
\begin{proof}
 Using repeatedly the last two propositions, we get the existence of a positive character $\mu$ of $P_{p-q+1}$ such that 
$Hom_{P_n\cap H_{p,q}}((\Phi^+)^{n-1}\Psi^+(1),\chi)\hookrightarrow Hom_{P_{p-q+1}\cap H_{p-q+1,0}}((\Phi^+)^{p-q}\Psi^+(1),\chi\mu)=
Hom_{P_{p-q+1}}((\Phi^+)^{p-q}\Psi^+(1),\chi\mu)$, and this last space is $0$ 
because $(\Phi^+)^{p-q}\Psi^+(1)$ and $\chi\mu$ are two non isomorphic irreducible representations of $P_{p-q+1}$, 
according to corollary 3.5 of \cite{BZ}.
\end{proof}

Another obvious corollary is the following proposition, which we will use in the next paragraph, to prove the 
functional equation of the exterior square $L$-function when $n$ is even.

\begin{cor}\label{homsteak}
Let $\rho$ be an irreducible representation of $G_{k}$ for $k\leq 2m-1$, and $\chi$ be a character of $H_{m,m}\cap P_{2m}$. Then the representation 
$\pi=(\Phi^+)^{2m-k-1}\Psi^+(\rho)$ is $(H_{m,m}\cap P_{2m},\chi)$-distinguished if and only if $\rho$ is $(H_{k/2,k/2},\chi\mu_{2m}^k)$-distinguished 
when $k=2p$ is even, and $(H_{(k+1)/2,(k-1)/2},\chi\mu_{2m}^k)$-distinguished when $k=2p-1$ is odd, where 
$\mu_{2m}^{2p}=\chi_{m,m}\chi_{m,m-1}\dots \chi_{p+1,p+1}$ and $\mu_{2m}^{2p-1}=\chi_{m,m}\chi_{m,m-1}\dots \chi_{p+1,p}$.
\end{cor}

Now we prove the first result about Levi subgroups that can distinguish a discrete series. The statement is the same as for cuspidal representations.

\begin{thm}\label{goodlevidiscr}
 Let $\Delta$ be a discrete series representation of $G_n$, $L$ be a maximal Levi sbgroup of $G_n$, and $\mu$ be a character of $L$, such that 
$\Delta$ is $\mu$-distinguished, then $n$ is even and $L\simeq M_{n/2,n/2}$. 
\end{thm}
\begin{proof}
As $L$ is conjugate to some $H_{p,q}$, for $p\geq q\geq 0$ (and $p+q=n$), it is sufficient to show that for any positive character $\mu$ of $H_{p,q}$, the discrete series $\Delta$ cannot be $(H_{p,q},\mu)$-distinguished if $p\geq q+1$. We write $\Delta=St_k(\rho)$ with $\rho$ a cuspidal representation of $G_r$, with $kr=n$. We are going to prove by induction on $k$ that $\Delta$ cannot be $(H_{p,q},\mu)$-distinguished when $p\geq q+1$. The case $k=1$ is Theorem 2.1 of \cite{M}. Now suppose that $\Delta$ is $(H_{p,q},\mu)$-distinguished, then the restriction 
$\Delta_{|P_n}$ is $(H_{p,q},\mu)$-distinguished, which implies, according to Proposition \ref{filtr}, that there is an integer 
$i$ in $\{1,\dots,k\}$ such that the representation $(\Phi^+)^{ir-1}\Psi^+(|.|^{i/2}St_{k-i}(\rho))$ is $P_n\cap H_{p,q}$-distinguished. 
But $i$ cannot be $k$ according to Corollary \ref{homzero}, and if $i>k$, then applying repeatedly Proposition \ref{rec}, we deduce that for 
some positive character $\chi$, the representation $\Psi^+(|.|^{i/2}St_{k-i}(\rho))$ of 
$P_{n-ir+1}$ is $(P_{n-ir+1}\cap H_{p-l,q-l},\chi)$-distinguished or $(P_{n-ir+1}\cap H_{p-l,q-l-1},\chi)$-distinguished for $l=[ir/2]$, 
depending on the parity of $ir$. This in turn implies that 
$St_{k-i}(\rho)$ is $(H_{p-l,q-l-1},|.|^{-i/2}\chi)$-distinguished or $(H_{p-l-1,q-l-1},|.|^{-i/2}\chi)$-distinguished, which is impossible by induction hypothesis. Thus $\Delta$ cannot be $(H_{p,q},\mu)$-distinguished if $p\geq q+1$.
\end{proof}

This statement doesn't extend without changes to generic representations, see for example \cite{P} or \cite{V}, where one sees that the $G_3$-modules $St_2(1)\times 1$ and $\rho\times 1$ are $M_{2,1}$-distinguished for every cuspidal $\rho$ of $G_2$ with trivial central character. The correct statement is the following in this case.

\begin{thm}\label{goodlevigen}
Let $\pi$ be a generic representation of $G_n$, $L$ be a Levi subgroup of $G_n$, and $\mu$ be a character of $L$ such that 
$\pi$ is $(L,\mu)$-distinguished, then $L\simeq M_{n/2,n/2}$ if $n$ is even, and $L\simeq M_{(n+1)/2,(n-1)/2}$ if $n$ is odd.
\end{thm}
\begin{proof}
We prove this result by induction again. If $\D=[\dots,\rho]$ is a discrete series with $\rho$ cuspidal, we denote by 
$e(\D)$ the real part of $\rho$'s central character (not that of $\D$'s central character). We will actually need to prove the following stronger statement: if $\pi$ is the representation $\D_1\times \dots \times \D_t$ of $G_n$, where 
$e(\D_{i+1})\geq e(\D_i)$, and if 
$\pi$ is $(L,\mu)$-distinguished, then 
$L\simeq M_{n/2,n/2}$ if $n$ is even, and $L\simeq M_{(n+1)/2,(n-1)/2}$ if $n$ is odd (notice that the $\D_i$'s always commute in a generic 
representation, hence they can always be ordered this way).\\
We start with $n=2$ and $3$. In the case $n=2$, the representation $\pi$ is $(G_2,\mu)$-distinguished if and only if the character $\mu$ 
is a quotient of $\pi$, but the only product of discrete series having $\mu$ as a quotient is $\pi =\mu|.|^{1/2}\times \mu|.|^{-1/2}$, which 
contradicts the hypothesis that $e(\D_{i+1})\geq e(\D_{i})$. The same argument works for the case $n=3$, because the only product of 
discrete series having $\mu$ as a quotient is $\pi = \mu|.|^{1}\times \mu|.|^0\times \mu|.|^{-1}$.\\
Now suppose that $n$ is even, and that $\pi=\D_1\times \dots \times \D_t$ is $(H_{p,q},\mu)$-distinguished for $p>q$, then again by Proposition 
\ref{filtr}, there would be a representation $\tau=\D_1^{(n_1)}\times \dots \times \D_t^{(n_t)}$ of $G_k$ with $k\leq n-1$, such that 
$(\Phi^+)^{n-k-1}\Psi^+(\tau)$ is $(H_{p-l,q-l-1},\mu)$-distinguished. The integer $k$ cannot be $0$ according to Corollary \ref{homzero}. 
If $k>0$, then 
again applying repeatedly Proposition \ref{rec}, we deduce that for some positive character $\chi$, the representation $\tau$ is 
 $(H_{p-l,q-l-1},\mu\chi)$-distinguished if $k$ is odd, or $(H_{p-l-1,q-l-1},\mu\chi)$-distinguished if $k$ is even, for some 
integer $l$. But this contradicts the induction hypothesis as the discrete series $\D_i^{(n_i)}$ still satisfy the property that 
$e(\D_{i+1}^{(n_{i+1})})\geq e(\D_i^{(n_i)})$ (as $\D^{(n)}$ is obtained by erasing the left end  of the "segment" $\D$). 
The case $n$ odd is treated similarly.
\end{proof}

\section{The functional equation of the local exterior square $L$-function, when $n$ is even}

Here $n$ is even and equals $2m$. A functional equation for the exterior square $L$-function is proved by a global argument, for generic 
representations appearing as local components of a cuspidal automorphic representation of the adelic points of $G_{2m}$ in \cite{KR}. 
We give a local proof, which works for all generic representations of $G_{2m}$. We first adapt the arguments of \cite{FJ} and \cite{JR}, 
to show that if $\pi$ is an irreducible representation 
of $G_n$, then there is an injection of the vector space $Hom_{P_n\cap S_n}(\pi,\Theta)$ into $Hom_{P_n\cap M_{m,m}}(\pi,1)$.\\

We will not discuss the normalisation of Haar measures in the following, but some choices have to be made for 
some of the integration formulas to be valid. We start with the following useful lemma.

\begin{LM}\label{utilLM}
Let $\pi$ be a representation of $P_{(m,m)}$, $L$ be an element of $Hom_{N_{m,m}}(\pi,\Theta)$, and $v$ belong the space of $\pi$. 
Denote by $S$ the map $p\mapsto L(\pi(p)v)$ on $P_{(m,m)}$ and by $\tilde{S}$ the map $g\mapsto S(diag(g,I_m))$ on $G_m$. 
Then there is $\xi$ 
in $\mathcal{C}^\infty_c(\mathcal{M}_m)$, such that for $g$ in $G_n$, one has $\tilde{S}(g)=\tilde{S}(g)\xi(g)$. 
In particular, the integral $c_k(S)=\int_{g\in G_m,|g|=q^{-k}}\tilde{S}(g)d^*g$ converges absolutely for all $k$ in $\Z$, and is zero for $k<<0$.
\end{LM}
\begin{proof}
Let $C$ be a compact open subgroup of $\mathcal{M}_m$, such that $v$ is invariant under $\begin{pmatrix} I_m & C \\ & I_m \end{pmatrix}$. Let $f$ 
be the indicator function of $C$. For a Haar measure $dx$ on $\mathcal{M}_m$, for any $p$ in $P_{(m,m)}$, we have 
$$S(p)=\int_{\mathcal{M}_m} S(p\begin{pmatrix} I_m & x \\ & I_m \end{pmatrix})f(x)dx.$$ 
Taking $p= diag(g,I_m) $, and using the left invariance of $S$, we obtain $\tilde{S}(g)=\tilde{S}(g)\xi(g)$, for 
$\xi(g)=\int_{\mathcal{M}_m} f(x)\Theta(gx) dx$, which is the Fourier transform of $f$. For the second part, one integrates on 
a compact part of $\mathcal{M}_m$ (the intersection of the support of $\xi$ and of $\{g\in M_m, |g|=q^{-k}\}$) which is a subset of $G_m$,
 hence on a compact of $G_m$. As the function $\tilde{S}$ is continuous, hence bounded on this set, the integral converges absolutely. 
It is of course zero as soon as $k$ is small enough for the support of $\xi$ not to meet the set $\{g\in G_m, |g|=q^{-k}\}$.
\end{proof} 

We keep on following \cite{FJ}, Section 3. The following lemma follows at once from smoothness considerations.
 
\begin{LM}\label{smoothLM}
Let $S$ belong to $Ind_{N_{m,m}}^{G_n}(\Theta)$, we can choose a smooth map $\phi$, with variables $u\in \mathcal{M}_m$, $b\in G_m$ and $k\in G_n(\o)$, 
and with compact support on $\mathcal{M}_n\times G_n \times G_n(\o)$, such that 
$$S(g)=\int_{u,b,k} S(g\begin{pmatrix} b^{-1} & \\ & 1\end{pmatrix}\begin{pmatrix} 1 & u \\ & 1\end{pmatrix}
\begin{pmatrix} 1 & \\ & b\end{pmatrix}k)\phi(u,b,k)|b|^md^*b^*dudk.$$
\end{LM}

Now we introduce a quite long list of notations. We denote by $\pi$ an irreducible representation of $G_n$, and denote by $L$ 
an element of $Hom_{N_{n,n}}(\pi,\Theta)$. Thanks to Frobenius reciprocity, we take $S$ in the image 
$S(\pi,L)$ of $\pi$ in the induced representation $Ind_{N_{m,m}}^{G_n}(\Theta)$. Lemma \ref{smoothLM} associates (non uniquely) to $S$ a map 
$\phi$ which is smooth with compact support in 
$\mathcal{M}_m\times G_m \times G_m(\o)$.\\
If $\phi'$ is the characteristic function of $K^{-1}$, where $K$ is the support of 
$\phi$ in the variable $b\in G_m$, we denote by $\Phi$ the map with variables 
$a\in \mathcal{M}_m, x\in \mathcal{M}_m, b \in G_m$ and $k\in G_n(\o)$, defined by 
$$\Phi(\begin{pmatrix} a & x \\ & b \end{pmatrix}, k))=\int_{u,v}\phi(u,b,k)\phi'(v)\Theta(ua-vx)dudv.$$ 
Defined this way, $\Phi$ is a smooth map with compact support for $(a,x,b,k) \in \mathcal{M}_m \times \mathcal{M}_m \times G_m \times G_n(\o)$, 
because of the properties of the Fourier transform. 
Let $\Omega$ be the open set of matrices $\begin{pmatrix} a & b \\ c & d \end{pmatrix}$ of $\mathcal{M}_n$, with $(c \ d)$ of rank $m$ 
in $\mathcal{M}_{m,n}$, and $\Omega_0$ the set of matrices $\begin{pmatrix} a & b \\ 0 & d\end{pmatrix}$, with $d$ invertible. The map 
$r:(p,k)\mapsto pk$ from $\Omega_0\times K$ to $\Omega$ is proper.
We then define $\Phi_*$ on the open set $\Omega$ on $G_n$ for $p\in P_{m,m}$, by 
$$\Phi_*(pk)=\int_{k'\in G_n(\o)\cap P_{m,m}} \Phi(pk'^{-1},k' k)dk'.$$ 
It is well defined, it has compact support in $\Omega$ as $r$ is proper. Finally, it is fixed by right multiplication under 
a compact open subgroup of $G_n(\o)$, hence because it has compact support in $\Omega$, it is smooth on $\Omega$. We extend it by zero outside 
$\Omega$, to obtain a smooth map $\Phi_*$ with compact support in $\mathcal{M}_n$. 
To finish this list of notations, we fix $U$ a subgroup of $G_n(\o)$, which fixes 
$\Phi_*$ by left multiplication of the variable, and denote by $S^U$ the map $\int_U \lambda(u) S du$ (where $\l(u)S(g)=S(u^{-1}g)$). It is a "matrix coefficient" of 
the representation $\pi$. We state a last lemma, before attacking the core of the problem.

\begin{LM}\label{lastLM}
For $S$, $\phi$ and $\Phi$ as above, the integrals $$I(S,\Phi,a,b)=\int_{x\in \mathcal{M}_n,k\in G_n(\o)}S(\begin{pmatrix} a & x \\ & b \end{pmatrix} k)
\Phi(\begin{pmatrix} a & x \\ & b \end{pmatrix},k)dx dk$$ and 
$$J(S,\phi,a,b)=\int_{u\in \mathcal{M}_n,k\in G_n(\o)}S(\begin{pmatrix} a &  \\ & 1\end{pmatrix}\begin{pmatrix} 1 & u \\ & 1 
\end{pmatrix}\begin{pmatrix} 1 &   \\ & b \end{pmatrix} k)\phi(u,b,k)du dk$$ both converge absolutely, and are equal. 
They define a map which is smooth and with compact support with respect to the variables 
$a\in \mathcal{M}_m$, and $b\in G_m$, and $k\in G_n(\o)$.
\end{LM}
\begin{proof}
Consider the map $f(a,x,b,k)=S(\begin{pmatrix} a & x \\ & b \end{pmatrix} k)\Phi(\begin{pmatrix} a & x \\ & b \end{pmatrix},k)$. Because 
of the invariance property of $S$, we have 
$f(a,x,b,k)=S(\begin{pmatrix} a &  \\ & b \end{pmatrix} k)\Phi(\begin{pmatrix} a & x \\ & b \end{pmatrix},k)\Theta(xb^{-1})$. 
By Fourier's inversion formula, and because of the choice of $\phi'$, 
we thus have 
$$\int_{x\in \mathcal{M}_n} f(a,x,b,k)dx =\int_{u\in \mathcal{M}_n} S(\begin{pmatrix} a &  \\ & b \end{pmatrix} k)\phi(u,b,k)\Theta(ua)du.$$
Finally, the relation 
$$S(\begin{pmatrix} a &  \\ & b \end{pmatrix} k)\Theta(ua)=S(\begin{pmatrix} a &  \\ & 1 \end{pmatrix}\begin{pmatrix} 1 & u \\ & 1 \end{pmatrix}
\begin{pmatrix} 1 &  \\ & b \end{pmatrix} k) $$ gives the expected relation. The assertion about smoothness and support are obvious.
\end{proof}

For $k$ and $l$ in $\Z$, we denote by $a_{k,l}(S,\Phi)$ the absolutely convergent integral 
$$a_{k,l}(S,\phi)=q^{-l m}\int_{|a|=q^{-k},|b|=q^{-l}}I(S,\Phi,a,b)d^*ad^*b.$$
We claim that the sum $\sum_{k,l} a_{k,l}(S,\Phi)q^{-ks}q^{-ls}$ is absolutely convergent for $Re(s)$ greater than a real $r_\pi$ depending 
only on $\pi$. 

\begin{prop}
The sum $\sum_{k,l} a_{k,l}(S,\Phi)q^{-ks}q^{-ls}$ is absolutely convergent for $Re(s)$ greater than a real $r_\pi$ depending 
only on $\pi$, it extends moreover meromorphically to a multiple of the $L$ function $L(\pi,s+1/2)$ (see \cite{GJ}), i.e. to an element of 
$L(\pi,s+1/2)\C[q^{\pm s}]$. 
\end{prop}
\begin{proof}
First, we observe that the integral $a_{k,l}(S,\Phi)$ is equal to $a_{k,l}(S,\Phi_*)$, with obvious notations. 
Moreover, $|a_{k,l}(S,\Phi_*)|\leq a_{k,l}(|S|,|\Phi_*|)$, in particular, for all fixed $j$, the finite sum 
$\sum_{k+l=j}|a_{k,l}(S,\Phi_*)||q^{-ks}q^{-ls}|$ is less than or equal to $\sum_{k+l=j}a_{k,l}(|S|,|\Phi_*|)|q^{-ks}q^{-ls}|$ which is itself equal to 
$$\int_{|g|=q^{-j}}|S(g)||\Phi_*(g)||g|^{Re(s)+m}d^*g=\int_{|g|=q^{-j}}|S^U(g)||\Phi_*(g)||g|^{Re(s)+m}d^*g$$ 
$$=\sum_{k+l=j}a_{k,l}(|S^U|,|\Phi_*|)|q^{-ks}q^{-ls}|.$$ But the sum $$\sum_{k,l} a_{k,l}(|S^U|,|\Phi_*|)|q^{-ks}q^{-ls}|$$ is equal to 
$\int_{G_n}|S^U(g)||\Phi_*(g)||g|^{Re(s)+m}d^*g$, which is known to be convergent for $Re(s)$ greater than some $r_\pi$, by 
\cite{GJ}. Morever, by [loc.cit] again, we know that $\sum_{k,l} a_{k,l}(S,\Phi)q^{-ks}q^{-ls}= \int_{G_n}S^U(g)\Phi_*(g)|g|^{s+m}d^*g$ extends 
meromorphically to a polynomial multiple of the $L$ function $L(\pi,s+1/2)$ on $\C$.
\end{proof}

Now we also know that $a_{k,l}(S,\Phi)$ is equal to $$b_{k,l}(S,\phi)= q^{-lm}\int_{|a|=q^{-k},|b|=q^{-l}}J(S,\phi,a,b)d^*a d^*b.$$ 
We are now able to prove the following seaked statement.

\begin{prop}\label{continuation}
The sum $I(S,s)=\sum_{k\in \Z} c_k(S) q^{-ks}$ converges absolutely for $Re(s)> r_\pi$, it defines a meromorphic function, which extends to $\C$ 
as a polynomial multiple of $L(\pi,s+1/2)$, which we still denote by $I(S,s)$.
\end{prop}
\begin{proof}
We write, for fixed $l$, $\sum_{k} b_{k,l}(S,\phi)q^{-ks}=\sum_{k} b_{k-l,l}(S,\phi)q^{-ks}q^{ls}$. 
We then write the equalities, where the sum over $l$ is finite as $J(S,\phi,a,b)$ has compact support in $b$:
$$\sum_{k,l} b_{k,l}(S,\phi)q^{-ks}q^{-ls}=\sum_{l} (\sum_{k} b_{k,l}(S,\phi)q^{-ks}) q^{-ls}$$
$$=\sum_{l} (\sum_{k} b_{k-l,l}(S,\phi)q^{-ks}q^{ls}) q^{-ls}= \sum_{l} (\sum_{k} b_{k-l,l}(S,\phi)q^{-ks})$$
$$\sum_{k} (\sum_{l} b_{k-l,l}(S,\phi))q^{-ks}= \sum_{k} c_k(S,\phi)q^{-ks},$$ 
for $c_k(S,\phi)$ the finite sum 
$$\sum_{l} b_{k-l,l}(S,\phi)=\int_{|a|=q^{-k}}\int_{b,u,k} S(\begin{pmatrix}a b^{-1} & \\ & 1\end{pmatrix}\begin{pmatrix} 1 & u \\ & 1\end{pmatrix}
\begin{pmatrix} 1 & \\ & b\end{pmatrix}k)\phi(u,b,k)|b|^md^*bdudkd^*a$$
This implies that one has $c_k(S,\phi)=\int_{|a|=q^{-k}}\tilde{S}(a)d^*a=c_k(S)$. The statement of the proposition follows.
\end{proof}

As a consequence, we obtain the injection between spaces of invariant functionals we were looking for.

\begin{prop}\label{mirabolicshalika}
Let $\pi$ be an irreducible representation of $G_n$, the vector space $Hom_{P_n\cap S_n}(\pi,\Theta)$ embeds as a subspace of 
$Hom_{P_n\cap M_{m,m}}(\pi,1)$.
\end{prop}
\begin{proof}
 Suppose that $L$ is a nonzero element of $Hom_{P_n\cap S_n}(\pi,\Theta)$ (the statement of the proposition being obvious if $Hom_{P_n\cap S_n}(\pi,\Theta)$ 
is reduced to zero). Then, thanks to Froebenius reciprocity, 
$\pi$ can be realised in $Ind_{P_n\cap S_n}^{G_n}(\Theta)$, we denote by $S(\pi,L)$ this realisation. We want to show that the rational map $I(S,s)$ 
defined in Proposition 
\ref{continuation} is nonzero for at least one $S$ in $S(\pi,L)$. Otherwise, this would mean that $c_k(S)$ is zero for all $S$. Replaceing $S$ by 
$\rho \begin{pmatrix} I_n & x \\ & I_n\end{pmatrix} S$ (where $\rho$ is the action by right translation), for $x$ in $\mathcal{M}_m$, 
this means that the integral $\int_{|g|=q-k} \tilde{S}(g) \Theta(xg) dg$ is zero for all $k$ and $x$. This means that the 
map $\1_{\{g\in G_m, |g|=q-k\}}\tilde{S}$, which belongs to $\sm_c(\mathcal{M}_m)$ according to Lemma \ref{utilLM}, has a Fourier transform 
which is null. We conclude that for all $k$,  the map $\1_{\{g\in G_n, |g|=q-k\}}\tilde{S}$ is zero, hence $\tilde{S}$ is zero on $G_m$.
 This is absurd 
as $L$ identifies with $S\mapsto \tilde{S}(I_m)$. In particular, according to Proposition 
\ref{continuation}, there is an element $Q$ in $\C(q^{-s})$, such that $Q(s)I(S,s)$ belongs to $\C[q^{-s}]$ for all $S$ 
in $S(\pi,L)$, and such that $Q(0)I(S,0)$ is nonzero for at least one $S$. The nonzero map $\Lambda:S\mapsto Q(0)I(S,0)$ then belongs to 
$Hom_{P_n\cap M_{m,m}}(\pi,1)$, and $L\mapsto \Lambda$ is the injection of the statement.
\end{proof}

This has as a consequence the functional equation of the square-exterior $L$-function in the even case. 
We will denote by $\pi^\vee$ the contragredient of the representation $\pi$ of $G_{2m}$. If $\pi$ is generic, and 
$W\in W(\pi,\theta)$, we will denote by $\tilde{W}$ the map $g\mapsto W(w_{2m} {}^t\!g^{-1})$, where $w_{2m}$ is the matrix 
corresponding to the permutation $i\mapsto 2m-(i-1)$, hence  $\widetilde{W}\in W(\pi^\vee,\theta^{-1})$. Finally if 
 $\phi$ belongs to $\sm_c(F^m)$, we will denote by $\widehat{\phi}$ its Fourier transform (choosing a measure which is self-dual with respect
 to $\theta$). 

\begin{thm}
 Let $\pi$ be a generic representation of $G_{2m}$. There exists an invertible element $\epsilon(\pi,\wedge^2,s)$ of $\C[q^s,q^{-s}]$, 
such that for every 
 $W$ in $W(\pi,\theta)$, and every $\phi$ in $\sm_c(F^m)$, one has the following functional 
equation: $$\epsilon(\pi,\wedge^2,s,\theta)\frac{J(W,\phi,s)}{L(\pi,\wedge^2,s)}=
\frac{J(\rho(w_{n,n})\widetilde{W},\widehat{\phi},1-s)}{L(\pi^\vee,\wedge^2,1-s)}$$ 
\end{thm}
\begin{proof}
In the following we put $n=2m$. We first prove that $dim(Hom_{P_{2m}\cap H_{m,m}}(\pi,|.|^s))\leq 1$ for all values of $q^{-s}$, 
except a finite number. Indeed, $\pi$ has a filtration with each factor of the form $(\Phi^+)^{n-k-1} \Psi^+(\tau)$ according to Proposition 
\ref{filtr}, for $k\leq n-1$, and 
where $\tau$ is a representation of $G_k$ admitting a central character. For every irreducible representation $\tau$ of 
$G_k$, for $k\geq 1$, we deduce that $Hom_{P_{2m}\cap H_{m,m}}((\Phi^+)^{n-k-1} \Psi^+(\tau),|.|^s))$ is zero except for a finite number of
 $q^{-s}$, as a consequence of Corollary \ref{homsteak} and the fact that 
$\tau$ has a central character. For all other values of $q^{-s}$, we deduce that 
$dim(Hom_{P_{2m}\cap H_{m,m}}(\pi,|.|^s))\leq dim(Hom_{P_{2m}\cap H_{m,m}}((\Phi^+)^{n-1} \Psi^+(\1),|.|^s))$ which is less than or equal to $1$ 
according to Corollary \ref{homsteak} again. Hence this proves our assertion 
about $dim(Hom_{P_{2m}\cap H_{m,m}}(\pi,|.|^s))$. Then Proposition \ref{mirabolicshalika} implies that 
$dim(Hom_{P_{2m}\cap S_{2m}}(\pi,|.|^{s}\Theta))\leq 1$ for all values of $q^{-s}$ except a finite number, with $|.|^s\Theta\begin{pmatrix}
p & x \\ &p\end{pmatrix}=|p|^{2s}\Theta(p^{-1}x)$. The Shalika subgroup $S_{2m}$ acts by $\begin{pmatrix}
g & x \\ &g\end{pmatrix}\phi(t)=\phi(tg)$ on $\sm_c(F^m)$, and it is easily verified that $B_s:(W,\phi)\mapsto \frac{J(W,\phi,s)}{L(\pi,\wedge^2,s)}$ 
and 
$\widetilde{B}_s:(W,\phi)\mapsto \frac{J(\rho(w_{n,n})\widetilde{W},\widehat{\phi},1-s)}{L(\pi^\vee,\wedge^2,1-s)}$ correspond to elements of 
$Hom_{S_{2m}}(\pi \otimes \sm_c(F^m),|.|^{-s}\Theta)$. Because of $\pi$'s central character, the space $Hom_{S_{2m}}(\pi,|.|^{-s}\Theta)$ is zero
 except 
for for a finite number of values of $q^{-s}$, hence the space $Hom_{S_{2m}}(\pi \otimes \sm_c(F^m),|.|^{-s}\Theta)$ is equal to
 $Hom_{S_{2m}}(\pi \otimes \sm_{c,0}(F^m),|.|^{-s}\Theta)$ except for those values ($\sm_{c,0}(F^m)$ is the subspace of $\sm_{c}(F^m)$ consisting 
of functions vanishing at zero). But 
$$Hom_{S_{2m}}(\pi \otimes \sm_{c,0}(F^m),|.|^{-s}\Theta)\simeq Hom_{S_{2m}}(\pi \otimes ind_{S_{2m}\cap P_{2m}}^{S_{2m}}(1),|.|^{-s}\Theta)$$
$$\simeq Hom_{S_{2m}}(\pi, Ind_{S_{2m}\cap P_{2m}}^{S_{2m}}(|.|^{-s+1}\Theta))\simeq Hom_{S_{2m}\cap P_{2m}}(\pi, |.|^{-s+1}\Theta)),$$
the first isomorhism identifying $S_{2m}\cap P_{2m}\backslash S_{2m}$ with $F^m-\{0\}$, and the last by Frobenius reciprocity law. As a 
consequence, we deduce that except for a finite number of $q^{-s}$, the dimension of $Hom_{S_{2m}}(\pi \otimes \sm_c(F^m),|.|^{-s}\Theta)$ 
is less or equal to $1$. This implies that 
$B_s$ and $\widetilde{B}_s$ are proportional up to an element of $\C(q^{-s})$. Taking $W_i$'s and $\phi_i$'s such that 
\[\sum_i \frac{J(W_i,\phi_i,s)}{L(\pi,\wedge^2,s)}=1,\] we see that $\epsilon(\pi,\wedge^2,s,\theta)$ belongs to 
$\C[q^{\pm s}]$, and taking $W_i$'s and $\phi_i$'s such that 
\[\sum_i \frac{J(\rho(w_{n,n})\widetilde{W_i},\widehat{\phi_i},1-s)}{L(\pi^\vee,\wedge^2,1-s)}=1,\] 
we see that it is a unit in this ring.
\end{proof}

\section{Relation between Shalika and local models}

 As we said before, any irreducible representation of $G_{2m}$ admitting a Shalika model, admits a local model.
The converse is true at least for cuspidal representations, it is proved in \cite{JNQ}, amongst many other properties of such representations. However, in this case, one can give a very short proof, which extends to a 
larger class of representations.\\
 We will say that an irreducible $M_{m,m}$-distinguished representation $\pi$ of $G_{2m}$ (with invariant linear form $L$) is relatively integrable (resp. relatively square-integrable) if there is (or equivalently for any) $v$ in the space of $\pi$, the relative coefficient  
$g\mapsto L(\pi(g)v)$ belongs to $L^1(M_{m,m}\backslash G_{2m})$ (resp. $L^2(M_{m,m}\backslash G_{2m})$). Cuspidal $M_{m,m}$-distinguished representations are relatively cuspidal (i.e., their relative coefficients belong to $\sm_c(M_{m,m}\backslash G_{2m})$), according to \cite{KT1}, and $M_{m,m}$-distinguished discrete series are relatively $L^2$ according to \cite{KT2}. The following theorem is a consequence of the principle of unfolding, explained in \cite{S}.

\begin{thm}\cite{S}\label{local-shalika}
 Let $\pi$ be a relatively integrable, or relatively square-integrable $M_{m,m}$-distinguished representation of $G_{2m}$, it admits a Shalika model. 
\end{thm}

\begin{proof}
The proof is that of Proposition 4.3 of \cite{S}. We give it here, using simpler tools due to the particular situation. We start with the $L^1$ case.\\ 
Let $L$ be a nonzero $M_{m,m}$-invariant linear form on the space $V$ of $\pi$. The map $\Phi: v \mapsto [g\mapsto L(\pi(g)v)]$ gives an injection of the  
$G_{2m}$-module $\pi$ into $L^1(M_{m,m}\backslash G_{2m})$.\\
 Call $\phi_v$ the function on $\mathcal{M}_m=\mathcal{M}_m(F)$ defined by 
$x \mapsto \Phi(v)\begin{pmatrix} I_m & x \\ 0 & I_m \end{pmatrix}$, it follows from Iwasawa decomposition that the functions $\phi_v$ belongs to 
$L^{1}(\mathcal{M}_m)$. \\
Indeed choose an open compact subgroup $U$ of the maximal compact subgoup $K_{2m}=G_{2m}(\o)$ of $G_{2m}$ leaving $\phi_v$ invariant by right translation, then one has 
$$ + \infty > \int_{M_{m,m}\backslash G_{2m}} |\Phi(v)(g)| dg=  \int_{N_{m,m} \times K_{2m}} |\Phi(v)(nk)| dn dk \geq Vol(U) \int_{N_{m,m}} |\Phi(v)(n)| dn .$$
 
This implies that the integral $\int_{N_{m,m}} |\Phi(v)(n)|dn= \int_{\mathcal{M}_m} |\phi_v(x)|dx$ is finite.\\
Because of the relations $\phi_{\pi \begin{pmatrix} I_m & x_0 \\ 0 & I_m \end{pmatrix} v}(x)= \phi_{v}(x+x_0)$, and $\phi_{\pi \begin{pmatrix} g_1 &  \\   & g_2 \end{pmatrix} v}(x)=\phi_{v}(g_2^{-1}xg_1)$, the map $v\mapsto \int_{\mathcal{M}_m} \phi_v(x)\theta(Tr(-x))dx$ will be a Shalika functional if it is nonzero. 
But if it was zero, replacing $v$ by $\pi \begin{pmatrix} I_m &  \\   & g_2^{-1} \end{pmatrix} v$, one would deduce that for any $v$, the integral $\int_{\mathcal{M}_m} \phi_v(x)\theta(Tr(-g_2x))dx$ is zero, i.e., 
the Fourier transform of $\phi_v$ would be zero on $G_m$, which is a Zariski open subset of $\mathcal{M}_m$. Hence 
$\phi_v$ (which is smooth) would be zero for any $v$, which is not possible as $\phi_v(0)=L(v)$.\\
For the $L^2$-case, according to \cite{S}, the $G_{2m}$-equivariant map 
$$U:\Phi\mapsto \tilde{\Phi}(g)=\int_{M_{2m}\cap S_{2m}\backslash S_{2m}}\Phi(sg)\theta^{-1}(s)ds$$ from 
$\sm_c(M_{2m}\backslash G_{2m})$ to $\sm(S_{2m}\backslash G_{2m},\theta)$ extends to a $L^{2}$ isometry between the 
$L^2(M_{2m}\backslash G_{2m})$ and 
$L^	2(S_{2m}\backslash G_{2m},\theta)$. Indeed, take $\Phi$ in $\sm_c(M_{2m}\backslash G_{2m})$, one has
$$\begin{aligned} \int_{M_{2m}\cap S_{2m}\backslash S_{2m}}\Phi(sg)\Theta^{-1}(s)ds = 
\int_{\mathcal{M}_m}\Phi(\begin{pmatrix} I_m & x \\ & I_m \end{pmatrix}g)\theta(-Tr (x))ds\end{aligned}.$$
Integrating over $M_{2m}\backslash G_{2m}\simeq N_{m,m}K_{2m}$ thanks to Iwasawa decomposition, and denoting $\Phi^K(g)=\int_{K_{2m}}\phi(gk)dk$ and $\phi^K(x)=\Phi^K\begin{pmatrix} I_m & x \\ & I_m \end{pmatrix}$, 
one gets $\begin{aligned}\int_{M_{2m}\backslash G_{2m}}|\Phi(g)|^2dg= \int_{\mathcal{M}_m}|\phi^K(x)|^2dg \end{aligned}$, but according to Parseval identity 
for $\phi^K$, this equals:
$$\begin{array}{l} \int_{\mathcal{M}_m}|\phi^K(x)|^2dg = \int_{\mathcal{M}_m}|\widehat{\Phi^K}(x)|^2dx =\int_{G_m}|\widehat{\phi^K}(x)|^2 dx =
\int_{G_m} \int_{\mathcal{M}_m}|\phi^K(y)\theta(Tr(-xy))dy|^2 dx \\
 =\int_{G_m}\int_{\mathcal{M}_m}|\phi^K(yx^{-1})\theta(Tr(-y))dy|^2 |x|^{-2m}dx =
 \int_{G_m}\int_{\mathcal{M}_m}|\phi^K(yx^{-1})\theta(Tr(-y))dy|^2 |x|^{-m}d^*x \\
 = \int_{G_m}\int_{\mathcal{M}_m}|\Phi^K(\begin{pmatrix} I_m & y \\ & I_m \end{pmatrix} \begin{pmatrix}x & \\ & I_m\end{pmatrix})\theta(Tr(-y))dy|^2 |x|^{-m}d^*x\\ = \int_{S_{2m}\backslash G_{2m}}\int_{\mathcal{M}_m}|\Phi(\begin{pmatrix} I_m & y \\ & I_m \end{pmatrix}g)\theta(Tr(-y))dy|^2d^*g=\int_{S_{2m}\backslash G_{2m}}|\tilde{\Phi}(g)|^2d^*g   \end{array}$$
The first equality on the last line is thanks to Iwasawa decomposition $G_{2m}=S_{2m}G_{m}K_{2m}$ again. But if $\D$ is relatively square integrable, then it is a submodule of the smooth part $L^{2,\infty}(M_{2m}\backslash G_{2m})$ of $L^2(M_{2m}\backslash G_{2m})$, hence $U(\pi)$ is a submodule of $L^{2,\infty}(S_{2m}\backslash G_{2m},\theta)$, and $\Phi\mapsto U(\Phi)(1)$ is a Shalika functional on $\D$.
\end{proof}

By \cite{BD}, representations of the type $\pi_s=\rho|.|^s\times \rho^\vee|.|^{-s}$ are $M_{n,n}$-distinguished for any irreducible representation $\rho$ of $G_n$, and any $s$ in $\C$. Let $\sigma$ be the involution $g\mapsto w_{2m}g w_{2m}^{-1}$ of $G_{2m}$. It can be checked by looking at the exponents of the representation $\pi_s$ on the Jacquet modules corresponding to $\sigma$-parabolic subgroups of $G_{2m}$ (it means the parabolic subgroups $P$ such that $\sigma(P)$ is opposite to $P$, see \cite{KT1} for this notion), that $\pi_s$ is relatively integrable 
for $Re(s)$ sufficiently large. Hence we see that the class of relatively integrable representations (to which our theorem applies) of $G_{2m}$ strictly contains the class of cuspidal representations of $G_{2m}$.\\

Of course, it is false that in general, irreducible $M_{m,m}$-distinguished representations of $G_{2m}$ 
have a Shalika model (for example the trivial representation of $G_2$). However, all infinite dimensional 
representations of $G_2$ (i.e., the generic representations) admit a local and a Shalika model at the same time, i.e., when their central 
character is trivial. It is indeed classical that a generic representation with trivial character of $G_2$ is $M_{1,1}$-distinguished, 
and in this case, a Shalika functional is nothing else than a Whittaker functional. Actually, in general, 
genric representations of $G_{2m}$ with a local model admit a Shalika model, according to a not yet 
published result of Wee Teck Gan. We briefly sketch the argument that he has communicated to us here, the result being part 
of his forthcoming paper \cite{G}. It is a consequence of the theta correspondance for the pair $(G_{2m},G_{2m})$, which 
gives an isomorphism between the space of linear periods of a representation $\pi$ of $G_{2m}$, and the space of the Shalika 
periods of the Theta lift $\Theta(\pi)$ of $G_{2m}$. But for generic representations, one can show that $\Theta(\pi)$ is 
irreducible, hence isomorphic to its unique irreducible quotient which is $\pi$ according to \cite{Mi}, and the answer follows. 
Actually it even gives a sufficient condition on $\Theta(\pi)$ (namely being irreducible) for the converse 
$``linear \ period \Rightarrow Shalika \ period''$ to be true. The Shalika periods for $G_{4}$ have been characterised in \cite{GT}.\\

\section{Discrete series of $G_{2m}$ with Shalika/local models}

First we recall a result of \cite{K}, characterizing discrete series of $G_{2m}$ admitting a Shalika 
model in terms of poles of the Jacquet-Shalika exterior square $L$-function.

\begin{prop}\label{poledist}{(Theorem 4.3. of \cite{K})}
 Let $n=2m$, and $\D$ be a square-integrable representation of $GL(n,F)$, then it admits a Shalika model if and only if 
$L (\D,\wedge^2,s)$ has a pole at zero.
\end{prop}

Now we express the exterior-square $L$ function of the Langlands parameter of $\D=St_k(\rho)$ of $G_{2m}$ in terms of the exterior and symmetric-square $L$-function of the Langlands parameter of the cuspidal representation 
$\rho$ of $G_r$ (with $kr=2m$).

\begin{prop}\label{Lgalois}
Let $\phi(\D)$ and $\phi(\rho)$ be the Langlands parameters of $\D=St_k(\rho)$ and $\rho$ respectively. One has the equality
 $$L(\phi(\D),s)=\prod_{i=0}^{[(k-1)/2]}L(\wedge^2,\phi(\rho),s+k-2i-1)\prod_{j=0}^{[k/2-1]}L(Sym^2,\phi(\rho),s+k-2j-2).$$
\end{prop}
\begin{proof}
 We recall that we denote by $\phi(\rho)$ the $r$-dimensional representation of the Weil group $W_F$ of $F$, which is the Galois parameter of $\rho$, and by $Sp(k)$ the (up to isomorphism) unique algebraic $k$-dimensional irreducible representation of $SL(2,\C)$. In this case, the Galois parameter $\phi(\D)$ of $St_k(\rho)$ is the representation 
$\phi(\rho)\otimes Sp(k)$ of $W'_F$. But it is then an exercise to see that the representation 
$\wedge^2(\phi(\rho))\otimes Sp(k))$ is isomorphic to 
$$\wedge^2(\phi(\rho))\otimes Sym^2( Sp(k)) \bigoplus Sym^2(\phi(\rho))\otimes \wedge^2( Sp(k)).$$
Moreover one shows that $Sym^2(Sp(k))\simeq \oplus_{i=0}^{[(k-1)/2]} Sp(2k-1-4i)$ and 
$\wedge^2( Sp(k))\simeq \oplus_{i=0}^{[k/2-1]} Sp(2k-3-4i)$.  Taking $L$-functions, this gives the formula: 
$$L(\phi(\D),s)=\prod_{i=0}^{[(k-1)/2]}L(\wedge^2,\phi(\rho),s+k-2i-1)\prod_{j=0}^{[k/2-1]}L(Sym^2,\phi(\rho),s+k-2j-2).$$

\end{proof}

Finally using Corollary 1.4 of \cite{KR}, we obtain the following theorem.

\begin{thm}\label{shaldiscr}
Let $\D$ be the representation $St_k(\rho)$ of $G_{2m}$.\\
If $k$ is odd, $\D$ has a Shalika (or a local) model if and only if $L(\rho,\wedge^2,s)$ has a pole at zero, or equivalently if and only if $\rho$ has Shalika (or a local) model.\\
If $k$ is even, $\D$ has a Shalika (or a local) model if and only if $L(\phi(\rho),Sym^2,s)$ has a pole at zero (where $\phi(\rho)$ is the Langlands parameter of $\rho$).
\end{thm}
 \begin{proof}
  According to Corollary 1.4 of \cite{KR}, the Rankin-Selberg exterior-square $L$-function of $\D$ coincides with the exterior-square $L$-function of its Langlands parameter. But we know from Proposition \ref{poledist}, that $\D$ has a Shalika model if and only if $L(\Delta,\wedge^2,s)$, hence $L(\phi(\Delta),\wedge^2,s)$, has a pole at zero. But because the central character of $\rho$ is unitary, we see that $L(\phi(\Delta),\wedge^2,s)$ has a pole at zero if and only if $k$ is odd and $L(\phi(\rho),\wedge^2,s)$ has a pole at zero, or $k$ is even, and 
$L(\phi(\rho),Sym^2,s)$ has a pole at zero. The statement of the theorem follows (keeping Corollary 1.4 of \cite{KR} in mind).
 \end{proof}

\end{document}